\documentclass[pamm,a4paper,fleqn
,final
]{w-art}
\usepackage{times,cite,w-thm}
\usepackage[T1]{fontenc}
\usepackage[utf8]{inputenc}
\usepackage[english]{babel}
\usepackage{graphicx}
\usepackage{hyperref}

\usepackage{amssymb}
\usepackage{tikz}
\usepackage{stmaryrd}
\usetikzlibrary{patterns}
\usetikzlibrary{arrows.meta}

\usepackage[
  authormarkup=none]{changes}

\newcounter{ADD}
\newcounter{DEL}
\def\ADDin{\addtocounter{ADD}{1}}
\def\ADDout{\addtocounter{ADD}{-1}}
\def\DELin{\addtocounter{DEL}{1}}
\def\DELout{\addtocounter{DEL}{-1}}

\makeatletter
\newcommand{\ifdraft}[2]{}
\ifthenelse{\boolean{Changes@optiondraft}}{%
  \renewcommand{\ifdraft}[2]{#1}}{%
  \renewcommand{\ifdraft}[2]{#2}}
\newcommand{\ifADD}[2]{\ifnum\value{ADD}>0{#1}\else{#2}\fi}
\newcommand{\ifDEL}[2]{\ifnum\value{DEL}>0{#1}\else{#2}\fi}
\makeatother

\newcommand{\delmarkup}[1]{%
  \DELin%
  \textcolor{red!80!black}{\stkout{#1}}%
  \DELout%
}
\newcommand{\addmarkup}[1]{%
  \ADDin%
  \ifDEL{}{%
    \textcolor{green!50!black}{#1}}
  \ADDout%
}
\ifdraft{%
  \newenvironment{addedenv}{%
    \ADDin%
    \color{green!50!black}}
  {%
    \ADDout%
    \color{black}}
}
{
  
}
\ifdraft{
  \setdeletedmarkup{\delmarkup{#1}}
  \setaddedmarkup{\addmarkup{#1}}
}{}

\ifdraft{
  \newcommand{\deleteX}[1]{%
    \DELin%
    {\color{red!80!black}{\sout{#1}}}%
    \DELout}%
}{
  \newcommand{\deleteX}[1]{}
}

\ifdraft{}{%
  \renewcommand{\added}[2][]{#2}
}

\begin{document}

\TitleLanguage[EN]
\title{DoD Stabilization of linear hyperbolic PDEs on general cut-cell meshes}

\author{\firstname{Gunnar} \lastname{Birke}\inst{1,}%
\footnote{Corresponding author: e-mail \ElectronicMail{g_birk01@wwu.de}}} 
\address[\inst{1}]{\CountryCode[DE]Applied Mathematics Münster, Münster University}
\author{\firstname{Christian} \lastname{Engwer}\inst{1,}%
}
\author{\firstname{Sandra} \lastname{May}\inst{2,}%
}
\address[\inst{2}]{\CountryCode[SE]Department of Information Technology, Uppsala University}
\author{\firstname{Florian} \lastname{Streitbürger}\inst{3,}%
}
\address[\inst{3}]{\CountryCode[DE]Department of Mathematics, Dortmund University}
\AbstractLanguage[EN]
\begin{abstract}
Standard numerical methods for hyperbolic PDEs require for stability a CFL-condition which implies that the time step size depends on the size of the elements of the mesh. On cut-cell meshes, elements can become arbitrarily small and thus the time step size cannot take the size of small cut-cells into account but has to be chosen based on the background mesh elements.

A remedy for this is the so called DoD (domain of dependence) stabilization for which several favorable theoretical and numerical properties have been shown in one and two space dimensions~\cite{emns,ms}. Up to now the method is restricted to stabilization of cut-cells with exactly one inflow and one outflow face, i.e. triangular cut-cells with a no-flow face (see \cite{emns}).

We extend the DoD stabilization to cut-cells with multiple in- and outflow faces by properly considering the flow distribution inside the cut-cell.
We further prove $L^2$-stability for the semi-discrete formulation in space and present numerical results to validate the proposed extension.
\end{abstract}
\maketitle                   

\section{Introduction}
To avoid the mesh generation process of complex geometries, cut-cell methods are an attractive alternative. The general idea is to start with a simple, e.g. structured, background mesh and to cut out the desired geometry. This results in a mesh with unstructured polyhedral cells, called cut-cells. Cut-cells can have an arbitrary shape and can become arbitrarily small, causing the small cell problem. To use explicit time stepping schemes for solving hyperbolic conservation laws, the time step size would need to be chosen based on the smallest cut-cell in the grid to ensure stability, which is in general not feasible.

Developing solution approaches to the small cell problem in the context of discontinuous Galerkin (DG) schemes is a very recent research branch, including for example the work in \cite{kronbichler, giuliani, Pei_Kreiss}. 
In this contribution we focus on the domain of dependence (DoD) stabilization, which was introduced in \cite{emns} for the linear transport equation in one and two space dimensions and was extended to non-linear systems in one space dimension in \cite{ms}. It is based on a DG scheme in space to allow for higher-order approximations and possesses several desirable theoretical properties. Numerical results show the expected higher-order behavior in smooth flow and robustness around shocks.

Up to now, the DoD stabilization in two dimensions has only been used to stabilize small triangular cut-cells for linear advection parallel to a ramp \cite{emns,ICOSAHOM}. In this setup, the small stabilized cut-cells have exactly one inflow and one outflow face, which was exploited in the design of the stabilization. When moving to non-linear or coupled linear problems, this does not hold true anymore and one has to deal with multiple inflow and outflow faces.

In this work, we take the first step in that direction by considering the linear advection equation on a cut-cell mesh with arbitrary flow directions, resulting in triangular cut-cells having 2 inflow and 1 outflow neighbor or reverse. As this causes significant additional complications, we will only consider piecewise constant polynomials here. We will prove $L^2$-stability for the semi-discrete scheme and present numerical results to validate the new extension of the stabilization terms.

The outline of the paper is as follows: we will first describe the problem setup and then introduce the new extended stabilization. Then we will show the $L^2$-stability proof and conclude with numerical results.

\section{Problem setup}
Let $\Omega \subset \mathbb{R}^2$ be an open and connected domain. We consider linear hyperbolic systems of the form
\begin{subequations}
\label{eq:hyper-system}
\begin{align}
u_t + Au_x + Bu_y & = 0 && \text{in } \Omega \times (0, T), \\
\tau u & = g && \text{on } \partial \Omega \times (0, T), \\
u & = u_0 && \text{on } \Omega \times \{t = 0\},
\end{align}
\end{subequations}
where $u(t) \in \mathbb{R}^m$, and $A, \, B \in \mathbb{R}^{m \times m}$ constant, and $\tau$ is an appropriate boundary operator such that we only impose inflow boundary conditions on incoming waves (and not everywhere on $\partial \Omega$). We require that for any unit vector $n\added{=(n_1,n_2)^T} \in \mathbb{S}^1$ the matrices $C = (n_1 A + n_2 B)$ are symmetric and simultaneously diagonalizable over the reals, i.e. there is an orthogonal matrix $O \in \mathbb{R}^{m \times m}$ and diagonal matrices $\Lambda_n \in \text{Diag}(\mathbb{R}^{m \times m})$ such that
$n_1 A + n_2 B = O \Lambda_n O^T ~ \forall n \in \deleteX{\mathbb{R}^2}\added{\mathbb{S}^1}$.

In our numerical tests we will choose $\Omega = [0, 1]^2$ and discretize it by a structured grid $\widehat{\mathcal{M}_h}$. We then introduce an artificial cut, a straight line going through the square, starting at $(x_0, 0)$ and having an angle $\gamma$ relative to the $x$-axis. This creates an internal boundary with two subdomains which we will resolve by a cut-cell mesh $\mathcal{M}_h$. A sketch is contained in Fig. \ref{fig:cutted-square-geometry}. So far we have always tested with flow parallel to that cut. Here, we consider flow in various directions, keeping the cut fixed.

\begin{figure}
\centering
\begin{tikzpicture}[color=black,semithick,scale=1.0]
\begin{scope}
\node at (0.0,1.3) {\footnotesize${\widehat{\mathcal{M}}}_h$};
\draw[step=0.25cm,color=gray] (-1,-1) grid (1,1);
\draw[color=gray] (-1,-1) -- (1,-1) -- (1,1) -- (-1,1) -- cycle;
\node at (1.5,0.0) {\huge$\cap$};
\end{scope}
\begin{scope}[xshift=3.0cm]
\node at (0.0,1.3) {\footnotesize${\bar{\Omega}}$};
\draw[color=gray] (-1,-1) -- (1,-1) -- (1,1) -- (-1,1) -- cycle;
\draw[color=gray] 
(-0.7999, -1.1)
--
(-0.7999, -1.0) node[anchor=north]{\color{black!70}\footnotesize$(x_0, 0)$} -- (1.0, 0.5102);
\draw[color=gray, dashed] (0.3, -1.0) arc(0:44:1.0);
\node[color=black!70] at (-0.1,-0.7) {$\gamma$};
\node at (1.5,0.0) {\huge$=$};
\end{scope}
\begin{scope}[xshift=6.0cm]
\node at (0.0,1.3) {\footnotesize${\mathcal{M}_h}$};
\draw[step=0.25cm,color=gray] (-1,-1) grid (1,1);
\draw[color=gray] (-1,-1) -- (1,-1) -- (1,1) -- (-1,1) -- cycle;
\draw[color=gray] (-0.7999, -1.0) -- (1.0, 0.5102);
\end{scope}
\begin{scope}[xshift=10cm]
\draw[dashed,color=gray] (-2.9, 0.5) -- (-0.9,0.8);
\draw[dashed,color=gray] (-2.9, 0.25) -- (-0.9,-0.8);
\draw[color=gray] (-0.8,-0.8) -- (0.8,-0.8) -- (0.8,0.8) -- node[above]{\footnotesize\color{black}$E \in \widehat{\mathcal{M}}_h$} (-0.8,0.8) -- cycle;
\draw[color=gray] (-0.8,-0.5) -- (0.75,0.8);
\node at (-0.4,0.4) {\footnotesize $E_1$};
\node at (0.2,-0.4) {\footnotesize $E_2$};
\end{scope}
\end{tikzpicture}
\centering
\caption{Construction of the mesh: Out of the structured grid $\widehat{\mathcal{M}_h}$ on the domain $\Omega$ the mesh $\mathcal{M}_h$ is constructed by introducing cut-cells $E_1, E_2 \subset E \in \widehat{M_h}$ along the cut such that $\bar{E_1} \cup \bar{E_2} = \bar{E}$.}
\label{fig:cutted-square-geometry}
\end{figure}
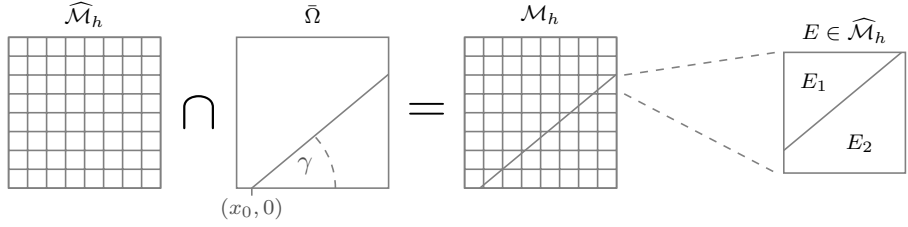

We define the sets of internal and external faces as
\[
\mathcal{F}^{\text{int}}_h =  \{ F = \partial E_1 \cap \partial E_2 | E_1, E_2 \in \mathcal{M}_h, \; E_1 \neq E_2, \; |F| > 0 \}, \quad
\mathcal{F}^{\text{ext}}_h =  \{ F = \partial E \cap \partial \Omega | E \in \mathcal{M}_h, \; |F| > 0 \},
\]
and the set of faces of an element $E \in \mathcal{M}_h$ by\added{ $\mathcal{F}_h^E = \{ F \in \mathcal{F}^{\text{int}}_h \cup \mathcal{F}^{\text{ext}}_h | F \subset \partial E \}$.}
\deleteX{\begin{displaymath}
\mathcal{F}_h^E = \{ F \in \mathcal{F}^{\text{int}}_h \cup \mathcal{F}^{\text{ext}}_h | F \subset \partial E \}.
\end{displaymath}}We choose a fixed local numbering on each of these sets and denote the neighbor element of $E$ corresponding to a face $F_i \in \mathcal{F}_h^E$ by $E_i$.
For the discretization in space we \deleteX{consider piecewise constant trial and test functions, which span the space}\added{choose the discrete function space}
\begin{displaymath}
\mathcal{V}^0_h(\mathcal{M}_h) = \{ v_h \in L^2(\Omega)^{\added{m}} \: | \: \added{(}v_h\added{)_i}{\vert_E} \in \mathcal{P}^0(E) \: \forall \added{(}E\added{, i)} \in \mathcal{M}_h \added{\times \{1, .., m \}} \}.
\end{displaymath}
For $v_h \in \mathcal{V}^0_h(\mathcal{M}_h)$ we denote by $v_h^E$ the value of $v_h$ on an element $E \in \mathcal{M}_h$. 

For interior faces $F \in \mathcal{F}^{\text{int}}_h$ we fix an orientation of the outer unit normal vector $n_F$ and denote the inner and outer element of $F$ by $E_1$ and $E_2$, respectively. 
We then define average and jump by
\begin{displaymath}
\textstyle
\left\lbrace\mskip-5mu\lbrace{u_h}\right\rbrace\mskip-5mu\rbrace:=\frac 1 2(u_h^{_{E_1}}+u_h^{_{E_2}}), \qquad
\left \llbracket{u_h}\right\rrbracket :=u_h^{_{E_1}}-u_h^{_{E_2}}.
\end{displaymath}
For exterior faces $F \in \mathcal{F}^{\text{ext}}_h$ we simply choose the unit outer normal \deleteX{and extend the definition of jump and average appropriately}\added{and denote by $u_h^{E_F}$ the solution on the cell that lies in the interior of the domain and contains face $F$}. 
We define the flux matrix on a face $F$ as
\begin{equation}
C_F = (n_F)_1 A + (n_F)_2 B = O \Lambda_F O^T,
\end{equation}
where $(n_F)_{1,2}$ denote the first and second component of the unit normal vector $n_F$ on face $F$.
Based on this, we define matrices which encode the flux directions as
\begin{displaymath}
C_F^+ = O \Lambda_F^+ O^T, \quad C_F^- = O \Lambda_F^- O^T
\quad\text{with}\quad
(\Lambda_F^+)_{i,i} = \max(0, (\Lambda_F)_{i,i})
\quad\text{and}\quad
(\Lambda_F^-)_{i,i} = \min(0,(\Lambda_F)_{i,i}).
\end{displaymath}
Note that $C_F = C_F^+ + C_F^-.$ We also introduce a generalization of the absolute value for such flux matrices by
$
|C_F| = C_F^+ - C_F^-.
$

The (unstabilized) upwind semi-discretization in space is then given as: Find $u_h(t) \in \mathcal{V}^0_h(\mathcal{M}_h)$ such that
\begin{equation}
(\partial_t u_h(t), v_h)_{L^2(\Omega)} + a^{\text{upw}}_h(u_h(t), v_h) + l_h(v_h) = 0 \quad \forall \: v_h \in \mathcal{V}^0_h(\mathcal{M}_h)
\end{equation}
with
\begin{align*}
a^{\text{upw}}_h(u_h, v_h) = & \sum_{F \in \deleteX{\Gamma^{\text{ext}}}\added{\mathcal{F}_h^{\text{ext}}}} \int_F \langle C_F^+ u_h\added{^{E_F}}, v_h\added{^{E_F}} \rangle ds + \sum_{F \in \deleteX{\Gamma^{\text{int}}}\added{\mathcal{F}_h^{\text{int}}}} \int_F \langle C_F \left\lbrace\mskip-5mu\lbrace{u_h}\right\rbrace\mskip-5mu\rbrace, \left \llbracket{v_h}\right\rrbracket \rangle + \frac{1}{2} \langle |C_F| \left \llbracket{u_h}\right\rrbracket, \left \llbracket{v_h}\right\rrbracket \rangle ds, \\
l_h(v_h) = & - \sum_{F \in \deleteX{\Gamma^{\text{ext}}}\added{\mathcal{F}_h^{\text{ext}}}} \int_F \langle C_F^- g, v_h\added{^{E_F}} \rangle ds.
\end{align*}
{Here, $\langle \cdot, \cdot \rangle$ denotes the scalar product in $\mathbb{R}^m$.} 
\added{We obtain $a^{\text{upw}}_h$ and $l_h$ by integration by parts, where the integral over internal edges leads to jump terms (second sum) and the boundary integral is split into outgoing waves (first sum) and incoming waves (right hand side).}
We then discretize in time using the explicit Euler scheme.

\section{Stabilization}

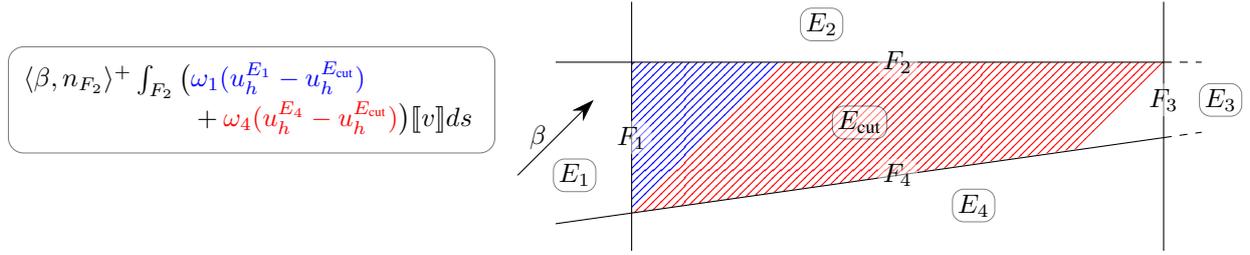
\begin{figure}[t]
\centering
\begin{tikzpicture}[whitebox/.style={%
        rounded corners=1.5mm,
        outer sep=0pt,
        inner sep=0.7mm,
        fill=white,
        opacity=0.5,
        text opacity=1},
        celltxt/.style={%
        rounded corners=1.5mm,
        outer sep=0pt,
        inner sep=0.7mm,
        draw, very thin,
        opacity=0.5,
        text opacity=1}
        ]
\fill[pattern=north east lines, pattern color=blue] (0.0, 2.0) -- (0.0, 0.0) -- (2.0, 2.0);
\fill[pattern=north east lines, pattern color=red] (0.0, 0.0) -- (2.0, 2.0) -- (7.0, 2.0) -- (5.85, 0.85);
\draw[black] (0,0) 
  -- node[whitebox]{$F_1$}
     node[celltxt,shift={(-.75,-.5)}]{$E_1$} 
     ++(0,2)
  -- node[whitebox]{$F_2$}
     node[celltxt,shift={(-1,.5)}]{$E_2$} 
     ++(7,0)
  -- node[whitebox]{$F_3$}
     node[celltxt,shift={(.75,0)}]{$E_3$} 
     ++(0,-1)
  -- node[whitebox]{$F_4$}
     node[celltxt,shift={(1,-.4)}]{$E_4$} 
     cycle;
\node[celltxt,whitebox] at (3, 1.2) {$E_{\text{cut}}$};
\draw (0, 2) -- (0, 2.8);
\draw (0, 2) -- (-1.0, 2);
\draw (0, 0) -- (0, -0.5);
\draw (0, 0) -- (-1.0, -0.142);
\draw (7, -0.5) -- (7, 1);
\draw (7, 2) -- (7, 2.8);
\draw[dashed] (7.0, 2.0) -- (7.5, 2.0);
\draw[dashed] (7.0, 1.0) -- (7.5, 1.071);
\draw[-{Stealth[length=3mm, width=2mm]}](-1.5, 0.5) -- node[left]{$\beta$}
++(1,1);
\node[
    rounded corners=2mm,
    outer sep=0pt,
    inner sep=2mm,
    draw=black!50
] at (-5, 1.5) {\parbox{6cm}{$\langle \beta, n_{F_2} \rangle^+ \int_{F_2} \big( \textcolor{blue}{\omega_1 (u_h^{E_1} - u_h^{E_{\text{cut}}})}\\
\phantom{\langle \beta, n_{F_2} \rangle^+ \int_{F_2}\big(}+\textcolor{red}{\omega_4 (u_h^{E_4} - u_h^{E_{\text{cut}}})}\big) \llbracket v \rrbracket ds$}};
\end{tikzpicture}
\caption{Domain of dependence extension illustrated on a four-sided cut-cell: We introduce a direct mass transport from cells $E_1$ and $E_4$ into $E_2$. The colored regions indicate the coupling between the faces, the corresponding parts of the stabilization term on $F_2$ are highlighted accordingly. From the graphic one can see that flow orientation and geometry play a central part in determining the right mass redistribution and the $\omega_i$ will have to be chosen accordingly. The face $F_3$ needs to be stabilized as well. Note that there should not be any coupling between between $F_2$ and $F_3$ as both are outflow faces.}
\label{fig:-dom-of-dep-ext}
\end{figure}

To deal with small cut-cells, an additional term $J_h^0 = \sum_{E_{\text{cut}} \in \mathcal{I}} J_h^{0, E_{\text{cut}}}$ with $\mathcal{I}$ being the set of small cut-cells that require stabilization is added to the space semi-discretization \added{(we will comment on our choice in the numerical results below)}. This results in the following scheme: Find $u_h(t) \in \mathcal{V}^0_h(\mathcal{M}_h)$ such that
\begin{equation}
\label{eq:semi-discrete-stabilization} (\partial_t u_h(t), v_h)_{L^2(\Omega)} + a^{\text{upw}}_h(u_h, v_h) + J^0_h(u_h, v_h) + l_h(v_h) = 0 \quad \forall \: v_h \in \mathcal{V}^0_h(\mathcal{M}_h).
\end{equation}
Let $E_{\text{cut}}$ be a small cut-cell that requires stabilization. The idea behind $J_h^{0, E_{\text{cut}}}$ is the following: When the time step size is not chosen to respect the small size of $E_{\text{cut}}$, the domain of dependence of an outflow neighbor $E$ of $E_{\text{cut}}$ will extend beyond $E_{\text{cut}}$. We therefore extend the numerical DoD of $E$ such that it receives information directly from the inflow neighbors of $E^{\text{cut}}$. The amount of mass passed directly between neighbors of $E_{\text{cut}}$ is chosen such that the update on $E_{\text{cut}}$ becomes stable.

To explain the main concept we consider the scalar linear transport equation $u_t + \nabla \cdot (\beta u) = 0$ with constant velocity $\beta\added{=(\beta_1,\beta_2)^T} \in \mathbb{R}^2$, i.e., we use $A=\beta_1$, $B=\beta_2$ in \eqref{eq:hyper-system}. 
Consider Fig. \ref{fig:-dom-of-dep-ext} where the domain of dependence of $E_2$ potentially reaches into $E_1$ and $E_4$. 
In that case, for choosing the time step based on the size of background cells, mass (physically) is moved from $E_1$ and from $E_4$ to $E_2$ in a single time step and this coupling must be mimicked by the stabilization.

We introduce an extension operator 
$
\mathcal{L}_{E'}^{\text{ext}}(u_h)(x) = u^{E'}_h(x)
$
for $E' \in \mathcal{M}_h$ which acts on functions $u_h \in \mathcal{V}^0_h(\mathcal{M}_h)$ and corresponds to evaluating 
the (constant) polynomial of cell $E'$ outside its original support.
Using this operator, the following stabilization term was introduced in \cite{emns} for triangular cut-cells with single outflow face $F_2$ and single inflow face $F_1$:
\begin{equation}
\label{eq:2facestab}
J^{0, E_{\text{cut}}}_h(u_h, v_h) = \eta_{E_{\text{cut}}} \int_{F_2} \langle \beta, n_{F_2} \rangle (\mathcal{L}_{E_1}^{\text{ext}}(u_h) - u_h^{\added{E_{\text{cut}}}}) \left \llbracket{v_h}\right\rrbracket ds.
\end{equation}
This term introduces a direct coupling between $E_1$ and $E_2$. The parameter $\eta_{E_{\text{cut}}} \in [0, 1]$ controls how much mass is transported via this coupling. In the case of exactly one inflow face $F_1$ and one outflow face $F_2$ (and one face with no-penetration b.c.), as considered in \cite{emns}, this term suffices to ensure stability. Here, however, the cut-cell is allowed to have two inflow faces as illustrated in Fig. \ref{fig:-dom-of-dep-ext}. Furthermore, these inflow faces will in general influence multiple neighbors of the cut-cell where the degree of influence depends on the geometry and flow direction. To handle this we propose the following extension of \eqref{eq:2facestab}
\begin{equation}
J^{0, E_{\text{cut}}}_h(u_h, v_h) = \eta_{E_{\text{cut}}} \sum_{F_j \in \mathcal{F}_h^{E_{\text{cut}}}} \int_{F_j} \sum_{F_i \in \mathcal{F}_h^{E_{\text{cut}}}} \omega_i \langle \beta, n_{F_j} \rangle^+ ( \mathcal{L}_{E_i}^{\text{ext}}(u_h) - u_h^{\added{E_{\text{cut}}}}) \left \llbracket{v_h}\right\rrbracket ds.
\end{equation}
Here the $\omega_i \in \mathbb{R}$ provide information about the flow distribution for incoming flow of the face $F_i$. Note that the extended solutions of all neighbor elements are evaluated on all faces, and $\omega_i
= 0$ if $F_i$ is not an inflow face. We provide a specific formula of how to choose these weights for triangular cut-cells below in \added{section \ref{sec:parameters}}\deleteX{\eqref{eq: def omega}}.

Going back to the system case \deleteX{(albeit still with the restriction of $AB = BA$) }we allow $\omega_i \in \mathbb{R}^{m \times m}$ and arrive at our final formulation
\begin{equation} \label{eq:final-stabilization-term}
J^{0, E_{\text{cut}}}_h(u_h, v_h) = \eta_{E_{\text{cut}}} \sum_{F_j \in \mathcal{F}_h^{E_{\text{cut}}}} \int_{F_j} \sum_{F_i \in \mathcal{F}_h^{E_{\text{cut}}}} \langle \omega_i C^+_F (\mathcal{L}_{E_i}^{\text{ext}}(u_h) - u_h^{\added{E_{\text{cut}}}}) , \left \llbracket{v_h}\right\rrbracket \rangle ds.
\end{equation}
{Note that in defining $J^{0, E_{\text{cut}}}$ we assume that all normal vectors $n_j$ for $F_j \in \mathcal{F}_h^{E_{\text{cut}}}$ correspond to \textit{outward} normal vectors with respect to $E_{\text{cut}}$.}
In order to ensure consistency and stability the weights $\omega_i$ must fulfill
\begin{align}
\label{eq:flux-normalization} \sum_{F_i \in \mathcal{F}_h^{E_{\text{cut}}}} \omega_i = & \, \text{Id}_{m \times m}, \\
\label{eq:flux-distribution} \sum_{F_j \in \mathcal{F}_h^{E_{\text{cut}}}} \int_{F_j} \omega_i C^+_j ds = & - \int_{F_i} C^-_i ds \qquad \forall \: F_i \in \mathcal{F}_h^{E_{\text{cut}}}.
\end{align}
Additionally we require that $\omega_i C_{F_j}^+$ is always symmetric and positive semi-definite. Equation \eqref{eq:flux-normalization} can be understood as an assurance that the overall amount of mass moved over a face by our stabilization is correct. {For the scalar case, we require the $\omega_i$ to build a convex combination.} Equation \eqref{eq:flux-distribution} means that a portion of the inflow is exactly redistributed over all outflow face candidates. This in particular prevents overshoots on small cut-cells for appropriate choices of $\eta_{E_{\text{cut}}}$.

\subsection{$L^2$-stability}

Equipped with the aforementioned properties of our stabilization we can show $L^2$-stability for the semi-discrete scheme in space. For \deleteX{simplicity}\added{brevity} we \deleteX{will choose}\added{consider} homogeneous inflow boundary conditions and assume that \deleteX{our solution has compact support in $\Omega$ for all times in $(0, T)$}
\added{exact and discrete solution vanish on the boundary,}
so that we can ignore any domain boundary terms and focus on the situation of cut-cells. \added{For $P^0$ functions this means that they are zero in all boundary cells.}

\begin{theorem}
Consider \eqref{eq:hyper-system} with homogeneous boundary conditions. Assume that the \added{discrete} solution $u\added{_h}(t)$ \deleteX{has compact support in $\Omega$}\added{vanishes on the boundary $\partial\Omega$} for all $t \in (0, T)$. Let $u_h(t) \in \mathcal{V}^0_h(\mathcal{M}_h)$ be the solution to the semi-discrete problem \eqref{eq:semi-discrete-stabilization}. Then it holds
\begin{displaymath}
||u_{\added{h}}(t)||_{L^2(\Omega)} \leq ||u_{\added{h}}(0)||_{L^2(\Omega)} \qquad \forall \; t \in (0, T).
\end{displaymath}
\end{theorem}
\begin{proof}
We choose $v_h = u_h(t)$ in \eqref{eq:semi-discrete-stabilization}. Any boundary terms vanish \deleteX{due to compactness assumption on $u_h$}\added{ as $\operatorname{trace}(u_h(t)) = 0$ on $\partial\Omega$}. This yields
\begin{displaymath}
(\partial_t u_h(t), u_h(t))_{L^2(\Omega)} + a^{\text{upw}}_h(u_h(t), u_h(t)) + J_h^0(u_h(t), u_h(t)) = 0.
\end{displaymath}
By the fundamental theorem of calculus
\begin{displaymath}
\int_0^t (\partial_{\tau} u_h(\tau), u_h(\tau))_{L^2(\Omega)} d \tau = \int_0^t \frac{d}{d{\tau}} \frac{1}{2} ||u_h(\tau)||_{L^2(\Omega)}^2 d\tau = \frac{1}{2} ||u_h(t)||^2_{L^2(\Omega)} - \frac{1}{2} ||u_h(0)||^2_{L^2(\Omega)}.
\end{displaymath}
To ease notation we will write $u = u_h(t)$ in the following. The goal now is to show that $a^{\text{upw}}_h(u, u) + J_h^0(u, u) \geq 0$.
We first consider $a^{\text{upw}}_h(u, u)$, which expands into
\begin{align*}
a^{\text{upw}}_h(u, u) = & \sum_{F \in \mathcal{F}_h^{\text{int}}} \int_F \langle  C_F \{\!\{u\}\!\}, \llbracket u \rrbracket \rangle + \langle  \tfrac{1}{2}|C_F| \llbracket u \rrbracket, \llbracket u \rrbracket \rangle ds
= \sum_{F \in \mathcal{F}_h^{\text{int}}} \int_F \langle C_F^+ u^{E_1} + C_F^- u^{E_2}, u^{E_1} - u^{E_2} \rangle ds \\
= & \sum_{F \in \mathcal{F}_h^{\text{int}}} \int_F \langle C_F^+ u^{E_1}, u^{E_1} \rangle - \langle C_F^+ u^{E_1}, u^{E_2} \rangle + \langle C_F^- u^{E_2}, u^{E_1} \rangle - \langle C_F^- u^{E_2}, u^{E_2} \rangle ds.
\intertext
{Now we add zeros (in form of $\pm \tfrac 1 2 \langle C_F^- u^{E_1}, u^{E_1} \rangle$ and 
$\pm \tfrac 1 2 \langle C_F^+ u^{E_2}, u^{E_2} \rangle$) to get}
= & \sum_{F \in \mathcal{F}_h^{\text{int}}} \int_F \frac{1}{2} \langle (C_F^+ - C_F^-) u^{E_1}, u^{E_1} \rangle + \frac{1}{2} \langle (C_F^+ + C_F^-) u^{E_1}, u^{E_1} \rangle - \langle C_F^+ u^{E_1}, u^{E_2} \rangle \\ & \phantom{\sum_{F \in \mathcal{F}_h^{\text{int}}} \int_F} + \langle C_F^- u^{E_{2}}, u^{E_{1}} \rangle 
+ \frac{1}{2} \langle (C_F^+ - C_F^-) u^{E_2}, u^{E_2} \rangle - \frac{1}{2} \langle (C_F^+ + C_F^-) u^{E_2}, u^{E_2} \rangle ds \\
= & \sum_{F \in \mathcal{F}_h^{\text{int}}} \int_F \frac{1}{2} \langle |C_F|(u^{E_1} - u^{E_2}), u^{E_1} - u^{E_2} \rangle + \frac{1}{2} \langle C_F u^{E_1}, u^{E_1} \rangle - \frac{1}{2} \langle C_F u^{E_2}, u^{E_2} \rangle \\
& \phantom{\sum_{F \in \mathcal{F}_h^{\text{int}}} \int_F} - \frac 1 2 \langle C_F u^{E_1}, u^{E_2} \rangle + \frac 1 2  \langle C_F u^{E_2}, u^{E_1} \rangle ds.
\end{align*}
Due to the symmetry of $C_F$, the terms in the last line cancel each other. For the last two terms in the second to last line we use the divergence theorem. Since $u_h(t)$ is elementwise constant, on $E \in \mathcal{M}_h$ it holds that
\[
0 = \int_E \nabla \cdot (\langle A u\added{^E}, u\added{^E} \rangle, \langle B u\added{^E}, u\added{^E} \rangle) dx = \sum_{F \in \mathcal{F}^{\text{int}}_h \cup \mathcal{F}_h^{\text{ext}}, F \cap \partial E \neq \emptyset} \int_F \langle C_F u\added{^E}, u\added{^E} \rangle ds,
\]
and therefore, these terms vanish as well. Finally, due to $|C_F|$ being positive semi-definite, we obtain positivity of $a^{\text{upw}}_h(u, u)$:
\[
a^{\text{upw}}_h(u, u) = \sum_{F \in \mathcal{F}_h^{\text{int}}} \int_F \frac{1}{2} \langle |C_F|(u^{E_1} - u^{E_2}), u^{E_1} - u^{E_2} \rangle ds \geq 0.
\]
We now investigate $J^0_h$.
For a small cut-cell $E_{\text{cut}} \in \mathcal{I}$ we have due to \eqref{eq:flux-normalization}
\begin{align*}
J_h^{0, E_{\text{cut}}}(u, u)
= & \eta_{E_{\text{cut}}} \sum_{F_j \in \mathcal{F}_h^{E_{\text{cut}}}} \int_{F_j} \langle (\sum_{F_i \in \mathcal{F}_h^{E_{\text{cut}}}}  \omega_i C_{F_j}^+ u^{E_i}) - C_{F_j}^+ u^{E_{\text{cut}}}, u^{E_{\text{cut}}} - u^{E_j} \rangle ds \\
= & \eta_{E_{\text{cut}}} \sum_{F_j \in \mathcal{F}_h^{E_{\text{cut}}}} \int_{F_j} \langle \sum_{F_i \in \mathcal{F}_h^{E_{\text{cut}}}} \omega_i C_{F_j}^+ u^{E_i}, u^{E_{\text{cut}}} \rangle -  \langle \sum_{F_i \in \mathcal{F}_h^{E_{\text{cut}}}} \omega_i C_{F_j}^+ u^{E_i}, u^{E_j} \rangle \\ & \phantom{\eta_{E_{\text{cut}}} \int_{F_j}} - \langle C_{F_j}^+ u^{E_{\text{cut}}}, u^{E_{\text{cut}}}\rangle + \langle C_{F_j}^+ u^{E_{\text{cut}}}, u^{E_j} \rangle ds.\\
\intertext{{Adding again zeros (in form of $\pm \langle \sum_{F_i \in \mathcal{F}_h^{E_{\text{cut}}}} \omega_i C_{F_j}^+ u^{E_i}, u^{E_i} \rangle$ and $\pm \langle C_{F_j}^+ u^{E_j}, u^{E_j} \rangle$) and reordering gives}}
= & - \frac{1}{2} \eta_{E_{\text{cut}}} \sum_{F_j \in \mathcal{F}_h^{E_{\text{cut}}}} \int_{F_j} \langle C_{F_j}^+ u^{E_{\text{cut}}}, u^{E_{\text{cut}}} \rangle - 2\langle C_{F_j}^+ u^{E_{\text{cut}}}, u^{E_j} \rangle + \langle C_{F_j}^+ u^{E_j}, u^{E_j} \rangle \\ & \phantom{- \frac{1}{2} \int_{F_j}} + \langle \sum_{F_i \in \mathcal{F}_h^{E_{\text{cut}}}} \omega_i C_{F_j}^+ u^{E_i}, u^{E_i} \rangle - 2 \langle \sum_{F_i \in \mathcal{F}_h^{E_{\text{cut}}}} \omega_i C_{F_j}^+ u^{E_i}, u^{E_{\text{cut}}} \rangle + \langle C_{F_j}^+ u^{E_{\text{cut}}} , u^{E_{\text{cut}}} \rangle \\ & \phantom{- \frac{1}{2} \int_{F_j}} - \langle \sum_{F_i \in \mathcal{F}_h^{E_{\text{cut}}}} \omega_i C_{F_j}^+ u^{E_i}, u^{E_i} \rangle + 2 \langle \sum_{F_i \in \mathcal{F}_h^{E_{\text{cut}}}} \omega_i C_{F_j}^+ u^{E_i}, u^{E_j} \rangle - \langle C_{F_j}^+ u^{E_j} , u^{E_j} \rangle ds \\
= & - \frac{1}{2} \eta_{E_{\text{cut}}} \sum_{F_j \in \mathcal{F}_h^{E_{\text{cut}}}} \int_{F_j} \langle C_{F_j}^+(u^{E_{\text{cut}}} - u^{E_j}), u^{E_{\text{cut}}} - u^{E_j} \rangle \\ & \phantom{- \frac{1}{2} \int_{F_j}} + \sum_{F_i \in \mathcal{F}_h^{E_{\text{cut}}}} \langle \omega_i C_{F_j}^+ (u^{E_i} - u^{E_{\text{cut}}}), u^{E_i} - u^{E_{\text{cut}}} \rangle \qquad \text{($C_{F_j}^+, \omega_i C_{F_j}^+$ symm. and \eqref{eq:flux-normalization})}
\\ & \phantom{- \frac{1}{2} \int_{F_j}} - \sum_{F_i \in \mathcal{F}_h^{E_{\text{cut}}}} \langle \omega_i C_{F_j}^+ (u^{E_i} - u^{E_j}), u^{E_i} - u^{E_j} \rangle ds\\
\overset{\eqref{eq:flux-distribution}}{=}& -\frac{1}{2} \eta_{E_{\text{cut}}} \hspace{-1ex}\sum_{F_j \in \mathcal{F}_h^{E_{\text{cut}}}} \int_{F_j} \langle |C_{F_j}|(u^{E_{\text{cut}}} - u^{E_j}), u^{E_{\text{cut}}} - u^{E_j} \rangle \\
& \phantom{- \frac{1}{2} \int_{F_j}}
- \hspace{-1ex}\sum_{F_i \in \mathcal{F}_h^{E_{\text{cut}}}} \langle \omega_i C_{F_j}^+ (u^{E_i} - u^{E_j}), u^{E_i} - u^{E_j} \rangle ds.
\end{align*}
Since $\eta_{E_{\text{cut}}} \in [0, 1]$ the first term inside the sum can be compensated with terms from $a_h^{\text{upw}}$. The second term is always non-negative since $\omega_i C_{F_j}^+$ is always positive semi-definite.
{Note that the second term corresponds to dissipation introduced by an extended jump.} This concludes the proof.
\end{proof}

\subsection{Choice of parameters} \label{sec:parameters}

To perform actual computations we need to select concrete $\omega_i$ in \eqref{eq:final-stabilization-term} that fulfill properties \eqref{eq:flux-normalization} and \eqref{eq:flux-distribution}. 
For the situation of linear simultaneously diagonalizable hyperbolic systems and triangular cut-cells we suggest \added{$\omega_i = |F_i| C_{F_i}^- (\sum_{F_k \in \mathcal{F}_h^{E_{\text{cut}}}} |F_k| C_{F_k}^-)^{-1}$ for each $F_i \in \mathcal{F}_h^{E_{\text{cut}}}$.}
For linear advection, this would result in $\omega_i=0$ for an outflow edge $F_i$ and $\omega_i$ corresponding to some sort of weighted proportion of the total inflow for an inflow edge $F_i$. 
%
We also need to set 
$\eta_{E_{\text{cut}}}$ for $E_{\text{cut}} \in \mathcal{I}$. A stable but not necessarily optimal choice is \added{$\eta_{E_{\text{cut}}} = ||\sum_{F \in \mathcal{F}_h^{E_{\text{cut}}}} \int_F \Lambda_F^- ds||_{\infty}$.}
\deleteX{
\begin{displaymath}
\eta_{E_{\text{cut}}} = ||\sum_{F \in \mathcal{F}_h^{E_{\text{cut}}}} \int_F \Lambda_F^- ds||_{\infty}.
\end{displaymath}
}

\section{Numerical results}
For the numerical tests we select angles $\gamma, \theta, \rho_1, \rho_2 \in [0, 2\pi)$ where $\gamma$ is the angle of the cut, see Fig. \ref{fig:cutted-square-geometry}, and set
\[
\Lambda_1 = \begin{pmatrix} \scriptsize \cos(\rho_1) & 0 \\ 0 & \cos(\rho_2) \end{pmatrix}, 
\; \Lambda_2 = \begin{pmatrix} \scriptsize
\sin(\rho_1) & 0 \\ 0 & \sin(\rho_2) \end{pmatrix},
\; O = \begin{pmatrix} \scriptsize
\cos(\theta) & -\sin(\theta) \\ \sin(\theta) & \cos(\theta) \end{pmatrix}.
\]
Then our system matrices will be given by $A = O\Lambda_1 O^T$ and $B = O\Lambda_2 O^T$.

The cut starts at $(x_0, 0) = (0.2001, 0)$. 
As initial conditions we choose
$
u_0(x) = O \begin{pmatrix} \scriptsize
\sin(2 \pi(x_1 \cos(\rho_1) + x_2 \sin(\deleteX{\rho1}\added{\rho_1}))) \\ \cos(2 \pi(x_1 \cos(\rho_2) + x_2 \sin(\deleteX{\rho2}\added{\rho_2}))) \end{pmatrix}
$.
Boundary conditions are given by the exact solution. Let $N$ denote the number of background cells in either coordinate direction.
We compute the time step size via
$
\Delta t = 0.4 \frac{h}{\max_{n \in \mathbb{S}^1}||n_1 \Lambda_1 + n_2 \Lambda_2||_{\infty}}
$, 
where $h = \frac{1}{N}$. The factor of $0.4$ allows that bigger cut-cells do not need to be stabilized. \added{We then choose $\mathcal{I} = \{E \in \mathcal{M}_h | \frac{|E|}{h^2} < 0.4 \}$ as the set of stabilized cut-cells.} The final time of our simulation is $T = 0.5$.

Our implementation is based on the DUNE framework (see \cite{dune1}, \cite{dune2}), in particular the dune-udg (see \cite{dune-udg}, \cite{dune-udg2}) and dune-pdelab modules. The local subtriangulations for the cut-cells are computed by the TPMC library (see \cite{tpmc}).

\begin{figure}
    \centering
    \includegraphics[scale=0.5]{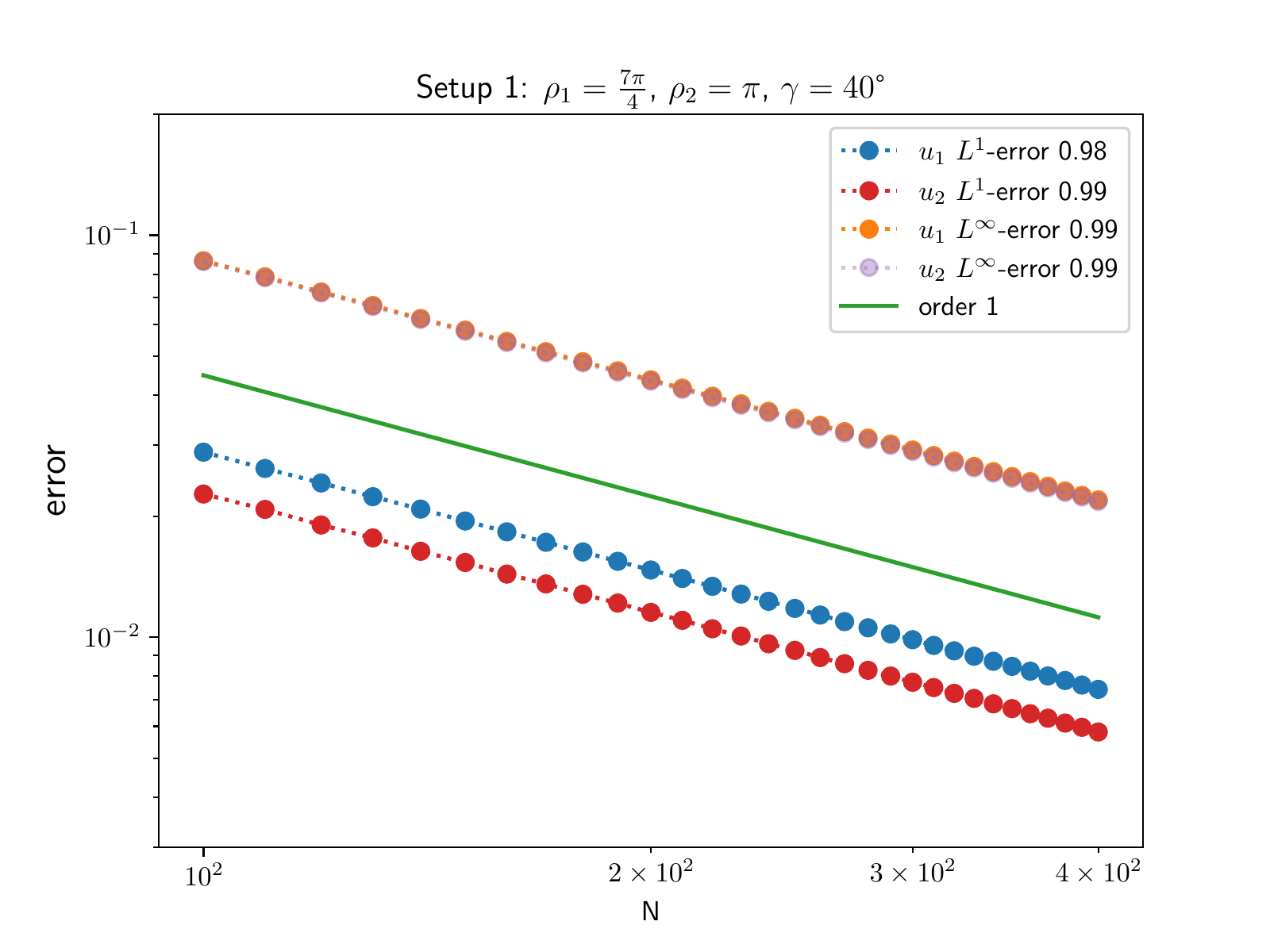}
    \includegraphics[scale=0.5]{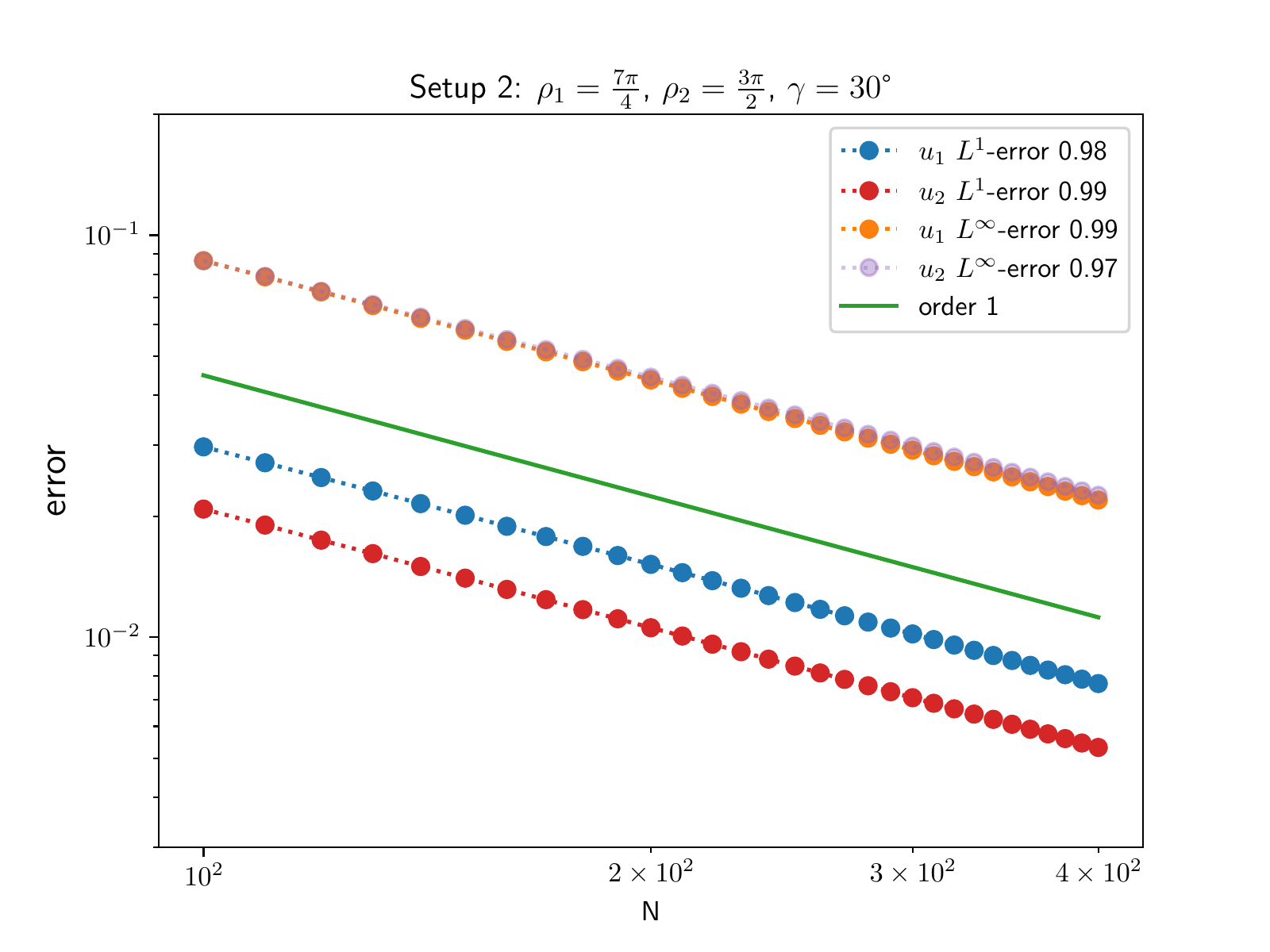}
    \caption{Plots of the error, measured in the $L^1$- (blue, red) and $L^{\infty}$-norm (orange, transparent purple). Left plot has values $\rho_1 = \frac{7 \pi}{4}$, $\rho_2 = \pi$ and $\gamma = 40$°, right plot has $\rho_1 = \frac{7 \pi}{4}$, $\rho_2 = \frac{3 \pi}{2}$ and $\gamma = 30$°. In all cases we have chosen $\theta = \frac{4\pi}{3}$. The green line is for reference.}
    \label{fig:numerical-results}
\end{figure}

Fig. \ref{fig:numerical-results} shows convergence plots for two particular setups. Note that the flow directions have been chosen to \textit{not} be parallel to the ramp angle $\gamma$. We observe the expected order of convergence for a first order scheme in both the $L^1$- and the $L^{\infty}$-norm. In addition, all numerically computed solution values, including those on small cut-cells, stayed within the bounds of the initial state during the simulation, confirming the added stability of the DoD stabilization.

\section{Discussion and Outlook}

We have extended the DoD stabilization to cut-cells with multiple inflow/outflow faces for the case of \added{component-wise and}
piecewise constant \deleteX{polynomials} \added{trial and test functions} and linear, simultaneously diagonalizable systems in two dimensions. We have proven $L^2$-stability for the semi-discrete setting. Numerically we observe full first-order convergence in different numerical tests and no over/undershoot on cut-cells. In future work, we plan to extend our method to more general systems, e.g., the acoustics and Euler equations. An extension of the presented formulation to higher-order approximations is ongoing research. 

\begin{acknowledgement}
  This work has been partially
  funded by the Deutsche Forschungsgemeinschaft (DFG, German Research Foundation) as project 439956613
  under contract numbers EN\,1042/5-1 and MA\,7773/4-1/2.
\end{acknowledgement}

\vspace{\baselineskip}

\end{document}               